\documentclass[12pt]{amsart}
\usepackage{amssymb,tikz,tikz-cd}
\usepackage{graphicx}
\numberwithin{equation}{section}

\begin{document}

{\theoremstyle{theorem}
    \newtheorem{theorem}{\bf Theorem}[section]
    \newtheorem{proposition}[theorem]{\bf Proposition}
    \newtheorem{conjecture}[theorem]{\bf Conjecture}
    \newtheorem{claim}{\bf Claim}[theorem]
    \newtheorem{lemma}[theorem]{\bf Lemma}
    \newtheorem{corollary}[theorem]{\bf Corollary}
\newtheorem{notation}[theorem]{\bf Notation}
}
{\theoremstyle{remark}
    \newtheorem{remark}[theorem]{\bf Remark}
    \newtheorem{example}[theorem]{\bf Example}
}
{\theoremstyle{definition}
    \newtheorem{definition}[theorem]{\bf Definition}
    \newtheorem{question}[theorem]{\bf Question}
}

 \newenvironment{dedication}
{\vspace{6ex}\begin{quotation}\begin{center}\begin{em}}
			{\par\end{em}\end{center}\end{quotation}}

\def\C{{\mathcal C}}
\def\height{\operatorname{ht}}
\def\Ass{\operatorname{Ass}}
\def\Min{\operatorname{Min}}
\def\depth{\operatorname{depth}}
\def\Edepth{\operatorname{Edge depth}}
\def\dim{\operatorname{dim}}
\def\Deg{\operatorname{Deg}}
\def\qed{\hfill$\Box$}
\newcommand{\m}{\mathfrak{m}}
\newcommand{\rar}{\rightarrow}

\title{Regular Edges, Matchings and Hilbert Series}
 
\author{Joseph Brennan}
\address{Department of Mathematics\\
University of Central Florida\\
4000 Central Florida Blvd.\\
Orlando, FL 32816-1364}
\email{joseph.brennan@ucf.edu}

\author{Susan Morey}
\address{Department of Mathematics \\
Texas State University\\
601 University Drive\\ 
San Marcos, TX 78666}
\email{morey@txstate.edu}


\begin{abstract}
When $I$ is the edge ideal of a graph $G$, we use combinatorial properities, particularly Property $P$ on connectivity of neighbors of an edge, to classify when a binomial sum of vertices is a regular element on $R/I(G)$. Under a mild separability assumption, we identify when such elements can be combined to form a regular sequence. Using these regular sequences, we show that the Hilbert series and corresponding $h$-vector can be calculated from a related graph using a simplified calculation on the $f$-vector, or independence vector, of the related graph. In the case when the graph is Cohen-Macaulay with a perfect matching of regular edges satisfying the separability criterion, the $h$-vector of $R/I(G)$ will be precisely the $f$-vector of the Stanley-Reisner complex of a graph with half as many vertices as $G$.
\end{abstract}

\maketitle
\bibliographystyle{amsalpha}

\section{Introduction}

The second half of the twentieth century saw the saw the rise of homological methods to address question in commutative algebra. 
One of the key elements in the introduction of these methods is the use of induction to reduce questions to the case where a given homological invariant is of minimal value.
The most ubiquitous application of this is to be found in the use of regular sequences. 
Such sequences allow for the iterated application of relations obtained from short exact sequences and hence induction on any invariant (such as the Krull dimension) that decreases on taking the quotient by a regular element.

Within the last fifty-five years there has also arisen a significant interest in combinatorial structures in  commutative algebra. 
Starting with the work of Hochster, Reisner and Stanley, much effort has been placed upon the recovery of information about ring invariants from combinatorial structures associated to the rings. 
In particular, starting with the pioneering efforts of Simis, Vasconcelos, and Villarreal there has been a focus on understanding the relationship between structural invariants and properties of graphs and rings associated to a graph.

Let \(G\) be a graph with vertex set \(V=V(G)\) and edge set \(E=E(G).\) 
Let \(k\) be a field. Associated to the graph \(G\) is the ideal 
\[I(G)=\left<\{xy\in k[V] \>|\> \{x,y\} \in E\}\right>\]
and the ring \(k[V]/I(G)\).

When studying graphs, there is also a method to employ induction. This involves using contraction and deletion of vertices or edges and studying how structural properties or invariants behave under such processes. 
This paper involves examining how these two paths of induction are related. For graphs, the process of contraction of an edge \(\{u,v\}\) is related to the process of taking the quotient by the element \(u-v\)  in \(k[V]/I(G).\)
We look at the question of when this element is a regular element. Having a regular sequence of linear binomial elements allows for an induction process to proceed on both the algebraic side and the graph theoretic side of the identification.
The utility of this idea is showcased in the ability to relate the Hilbert series of one graph to structural invariants of a simpler smaller graph obtained by edge contractions.

In Section 2 we introduce terminology and notation that are critical to this paper.  Property \(P\) is the primary graph property that we use in this paper. 
Property \(P\) was originally a property of perfect matchings of graphs. We define it as a property of an edge of the graph. 
We further give a history of the development and use of Property \(P\) in understanding perfect matchings and its sometimes hidden use in commutative algebra.

Section 3 provides the reason that Property \(P\) is central to this paper. We define an edge \(\{u,v\}\) to be a {\em regular edge} in a way that is equivalent to the element \(u-v\) being a regular element of the ring \(k[V]/I(G).\)
We then show in Theorem \ref{good=P} that a edge being a regular edge is equivalent to the edge having Property \(P.\)

To proceed in parallel in induction between graphs and rings we must address graphs with loops. 
Section 4 introduces Property \(P\) for graphs with loops and shows in Theorem \ref{good=P loops} the analogue of Theorem \ref{good=P} for graphs with loops.
This is then used in Theorem \ref{reg seq} in Section 5 to connect (not necessarily perfect) matchings whose edges have Property \(P\)  and a separability condition to regular sequences on the ring \(R/I(G).\)

Section 6 asks and answers the question of the existence and characterization of linear binomial regular elements that are not associated with an edge. Theorem \ref{binomial regular} connects those elements with Property \(P.\)

The preceding results are then used in Section 7 to compute the Hilbert series of the ring associated to a graph by reduction as established in \cite{Brennan-Morey}. 
The paper ends with some illustrative examples that we hope will provide the reader with considerable interest in the methods developed in this paper.

\section{Background and Property \(P\)}\label{background}

In this section we establish notation and provide graph theoretic terminology and definitions that will be used throughout. For further background and any unexplained terms we refer the reader to \cite{Chartrand-book,West-book}.

Let \(G\) be a graph with vertex set \(V=V(G)=\{x_1, \ldots, x_n\}\) and edge set \(E=E(G).\) 
Let \(k\) be a field. The ring \(R=k[x_1, \ldots, x_n],\) sometimes denoted as \(k[V]\) in the literature, is the polynomial ring over \(k\) whose indeterminates are identified with the vertices of the graph. 
Associated to the graph \(G\) is the ring \(R/I(G)\) where \[I(G)=\left<\{xy\in R \>|\> \{x,y\} \in E\}\right>.\]

A \emph{vertex cover} of the graph \(G\) is a collection of vertices \(X\subseteq V(G)\) such that every edge of \(G\) contains a vertex in \(X.\)
The height of the edge ideal, \(\height(I(G))\) is the cardinality of a minimum vertex cover of \(G\), which Villareal \cite{MonomialAlgebras}, following Berge \cite{Berge}, denotes by \(\alpha_0\), as in some of the combinatiorial optimization literature. This is also denoted by \(\beta\) in some of the graph theoretic literature \cite[Definition 3.1.12]{West-book}. Further, in \cite[(3.35)]{CombinatorialOptimizationA}, Schrijver denotes this by \(\tau(G)\).  Since the notation varies and our perspective is algebraic, we will use the height notation. By abuse of notation, we will refer  to this invariant as the height of \(G\), defining 
\(\height(G) = \height(I(G))\) in the ring \(R.\)

An \emph{independent set of vertices} of a graph \(G\) is a set of vertices \(Y\subseteq V(G)\) such that no edge has both of its vertices in \(Y.\)  
A \emph{matching} of \(G\) is a collection of edges \(M\subseteq E(G)\) such that if \(e_1,e_2 \in M\) then \(e_1\) and \(e_2\) do not share a common vertex. This is also called an independent set of edges in the literature. The collection of matchings is ordered by inclusion. The largest size of a maximal matching of the graph \(G\) is the matching number which we will denote by 
\(\hbox{\rm mat}(G)\) . The matching number appears in the graph-theoretic literature \cite[Definition 3.1.12]{West-book} as  \(\alpha'(G)\), by Villarreal \cite[Definition 7.17]{MonomialAlgebras} 
as \(\beta_1(G)\) and by Schrijver \cite[(3.35)]{CombinatorialOptimizationA} as \(\nu(G).\)

A \emph{perfect matching} is a matching \(M\) in which every vertex appears in some edge of \(M\). If a graph \(G\) has a perfect matching, then \(G\) does not contain isolated vertices and the cardinality of the matching is precisely one-half the number of vertices of \(G\). 

A graph \(G\) is \emph{well-covered} if every minimal vertex cover of \(G\) has the same cardinality. This is equivalent to every minimal prime, and thus every associated prime, of \(I(G)\) having the same height. Thus \(G\) is well covered if and only if \(I(G)\) is unmixed. If in addition, \(G\) does not have isolated vertices and every minimal vertex cover has precisely \(|V|/2\) elements, then \(g\) is \emph{very well-covered}. 

\medskip

An important concept that has appeared in the study of graphs is  called Property \(P.\) Although first defined for perfect matchings, we give a more general definition focused on single edges here.

\begin{definition}\label{Property P}
Let \(G\) be a simple graph and \(e=\{x,y\} \in E(G).\)  Then \(e\) is said to have {\em Property P} if for any \(a,b \in V(G)\) with \(\{a,x\}, \{y,b\} \in E(G)\) we have \(a \neq b\) and \(\{a,b\} \in E(G).\)
\end{definition}

A perfect matching has Property \(P\) if every edge in the matching has Property \(P\).
The idea of Property \(P\) appeared in the thesis of Staples \cite{StaplesThesis} and independently in \cite{Ravindra}. Recall that the neighbor set of a vertex \(v\) is \( N_G(v)=\{ u \mid \{u,v\} \in E(G)\},\) and the closed neighbor set of a vertex \(v\) is \( N_G[v]=N_G(v) \cup \{v\}.\) Note that Property \(P\) can also be described by the requirement that the induced graph on the neighbor sets of the vertices of an edge with Property $P$, excluding the vertices in the edge, contains the complete bipartite graph on those two vertex sets.

Property \(P\) is closely related to other graph theoretic properties. For instance, the well-covering property is closely related to Property \(P\).

\begin{theorem} \cite{Ravindra,Staples} A bipartite graph without isolated vertices is well covered if and only if there is a perfect matching with Property \(P.\)\end{theorem}

Favaron is the first to explicitly coin the term Property \(P\) as a property of perfect matchings. This was used to extend the above result to graphs without isolated vertices, where for graphs that are not bipartite, very well covered is used to replace well covered. This connection was also  independently shown by Staples.

\begin{theorem} \cite{Favaron,Staples} A graph without isolated vertices is very well-covered if and only if there exists a perfect matching that satisfies Property \(P.\)\end{theorem}

Later Levit and Mandrescu gave an additional criterion for graphs to be very well-covered. Let \(mat(G)\) be the size of a maximal matching of \(G\).

\begin{theorem} \cite{Levit} A graph without isolated vertices is well covered if and only if 
\[ |V(G)| - \height(G)=|E(G)| - mat(G).\]
\end{theorem}

Later Rautenbach and Volkmann \cite{R-V} directly proved the equivalence of the Levit and Mandrescu criterion with the existence of a perfect matching with Property \(P.\)


Herzog and Hibi used an ordering criterion on the vertices of a bipartite graph to classify Cohen-Macaulay-bipartite graphs. 

\begin{theorem}\cite[Theorem 3.4]{H-H}\label{bipartite}

Let \(G\) be a finite bipartite graph on the vertex set \[V(G)= \{x_1,\dots, x_n\} \cap\{y_1,\dots, y_n\}\]
and suppose that 
\begin{itemize} 
\item for all \(1\leq i\leq n\) one has \(\{x_i,y_i\}\in E(G)\)
\item  if \(\{x_i,y_j\}\in E(G)\) then \(i\leq j\) 
\end{itemize}
Then \(G\) is Cohen-Macaulay if and only if the condition
\[\hbox{\rm If }  \{x_i,y_j\}\in E(G)  \, \hbox{\rm  and } \{x_j, y_k\}\in E(G) \, \hbox{\rm  with } i<j<k \, \hbox{\rm  then } \{x_i, y_k\}\in E(G)\]
holds.
\end{theorem} 

Note that although the language used is different, the final condition above is equivalent to Property \(P\). Using the language of Theorem~\ref{bipartite}, in \cite[Theorem 1.1]{Villarreal}, inspired by the above result of Herzog and Hibi, Villarreal proved that for a bipartite graph \(G\) with a perfect matching, the edges of the perfect matching all satisfy Property \(P\) if and only if \(G\) is unmixed.

Further elaboration by Castrilli\'{o}n, Cruz, and Reyes \cite[Proposition 15]{CCR} showed that a K{\"o}nig graph without isolated vertices is unmixed provided that a matching of K{\"o}nig type has Property \(P.\) 
These ideas were extended to directed graphs in \cite{Pitones} and 
these ideas were then further extended in the oriented case \cite{Cruz,CruzReyes,CruzThesis}.

The following result of \cite{R-V} is a reformulation of the theorem of  \cite{Favaron,Staples} linking Property \(P\) to very well-covered graphs.

\begin{proposition}\cite[Corollary 18]{R-V}
Let \(G\) be a graph without isolated vertices. Then \(G\) is very well-covered of an only if
\begin{enumerate}
\item there exists a perfect matching \(F\) with \(\height(G)= |F| =|V(G)|/2.\) and
\item for each edge \(e=\{x,y\}\in F,\) if \(S\subset V(G)\setminus \{x,y\}\) with \(|S| \leq 2\), then \(\{x,y\} \not\subseteq N_G(S).\)
\end{enumerate}
\end{proposition}

\section{Regular edges}\label{regular}

In this section, we show that edges that satisfy Property \(P\) can be used to form regular elements on $R/I(G)$. We first note that Property \(P\) can be reinterpreted as a property of  induced subgraphs containing the edge.

\begin{lemma}\label{induced}
Let \(G\) be a graph and \(e=\{x,y\} \in E(G).\) Then \(e\) has Property \(P\) if and only if for every \(a \in N_G(x)\setminus \{y\}\) and \(b \in N_G(y)\setminus \{x\}\), the induced subgraph on \( \{a,x,y,b\} \) is a \(4\)-cycle.
\end{lemma}

\begin{proof}
This follows directly from the definition of Property \(P\).
\end{proof}

We now define the key concept of regular edge.

\begin{definition}\label{regular}
An edge \(e=\{a,b\}\) of a graph \(G\) is a {\em regular} edge if for every associated prime \(Q\) of \(I(G),\)  \(e \not\in Q^2.\) 
That is, precisely one of \(a\) or \(b\) is in \(Q,\) but not both. 
\end{definition}

The reason for this nomenclature is clear from the following result.

\begin{lemma}\label{regular sums}
Let \(G\) be a graph and \(\{x,y\}\in E(G)\). Then the following are equivalent:
\begin{itemize}
\item[\((i)\)] \(\{x,y\}\) is a regular edge of \(G\),
\item[\((ii)\)] \(x+y\) is a regular element of the ring \(R/I(G)\),
\item[\((iii)\)] \(x-y\) is a regular element of the ring \(R/I(G)\).
\end{itemize}
\end{lemma}

\begin{proof}
Assume \(\{x,y\}\) is a regular edge of \(G\). Then by definition, either \(x\) or \(y\) is in \(Q\) for every associated prime \(Q\) of \(R/I(G)\) but not both. Thus
 \[x+y \not\in \bigcup_{\scalebox{0.6}{%
 $Q \in \Ass (R/I(G))$}} Q\]
and so \(x+y, x-y\) are regular on \(R/I(G)\). Thus \((i) \Rightarrow (ii)\). The argument that \((i) \Rightarrow (iii)\) follows similarly.

Assume \(x+y\) is a regular element of \(R/I(G)\). Then \(x+y \not\in Q\) for every \(Q \in \Ass(R/I(G))\). Since \(\{x,y\} \in E(G)\) and every associated prime contains a minimal prime, which corresponds to a minimal vertex cover of \(G\), then at least one of \( x,y \in Q\). If both \( x,y \in Q\), then \( x+y \in Q\), a contradiction. Thus precisely one of \( x,y \in Q\) and \( xy \not\in Q^2\). Hence \(\{x,y\}\) is a regular edge of \(G\) and \((ii) \Rightarrow (i)\). The proof that \((iii) \Rightarrow (i)\) is similar.
\end{proof}

In order to establish a relationship between Property \(P\) and regular edges, a first step is to explore potential induced graphs containing a regular edge. We first show that a regular edge cannot be contained in an induced \(3\)-cycle.

\begin{lemma}\label{no-triangle}
If \(G\) is a graph and \(\{x,y\}\) is a regular edge of \(G,\) then \(\{x,y\}\) is not contained in any triangle in \(G.\)
\end{lemma}

\begin{proof}
Suppose \(b \in V(G)\) is such that \(x,y,b\) form the vertices of a triangle in \(G.\) Set \(W=V(G)\setminus \{b\}.\) 
Then \(W\) is a vertex cover of \(G.\) 
Thus there exists \(Q \subseteq W\) such that \(Q\) is a minimal vertex cover of \(G.\) Since \(b \not\in W,\) then \(x\in Q\) because \(\{x,b\}\in E(G)\) and \(y \in Q\) because \(\{y,b\}\in E(G).\)
This contradicts the definition of \(\{x,y\}\) being a regular edge. Thus no such \(b\) exists.
\end{proof}

We now consider potential induced subgraphs using distinct neighbors of the vertices of a regular edge.

\begin{lemma}\label{no-3-path}
If \(G\) is a graph and \(\{x,y\}\) is a regular edge of \(G,\) then for any \(a,b \in V\) with \(a \in N_G(x)\setminus \{y\}\) and \(b \in N_G(y) \setminus \{x\},\) the induced subgraph of \(G\) on the vertices \(a,x,y,b\) is a \(4\)-cycle. That is, if \(a \in N_G(x)\) and \(b\in N_G(y),\) then \(\{a,b\}\in E(G).\) 
\end{lemma}

\begin{proof}
First, notice that if \(a = b,\) or  \(\{a,y\}\in E(G),\) or \(\{x,b\}\in E(G)\) then there is a triangle of \( G\) containing \( \{x,y\}\), contradicting Lemma~\ref{no-triangle}. Thus \(a \ne b,\) and  \( \{a,y\}\not\in E(G),\) and \(\{x,b\}\not\in E(G).\) 
Then if $\{a,b\}\not\in E(G)$, set $W=V\setminus \{a,b\}$. Since $W$ is a vertex cover of $G,$ it contains a minimal vertex cover $Q$. 
By definition, $x \in Q$ since $a \not\in Q$ and $\{x,a\}\in E(G)$ and $y\in Q$ since $b\not\in Q$ and $\{y,b\}\in E(G)$. 
Thus $x,y \in Q$, a contradiction. 
Hence $\{a,b\} \in E(G)$ and  the induced subgraph of \( G\) on \(a,c,y,b\) is a \(4\)-cycle.
\end{proof}

Using this information about induced subgraphs, we now show that Property \(P\) is equivalent to regularity for an edge.

\begin{theorem} \label{good=P}
Let \(G\) be a graph and \(e\) an edge of \(G.\) Then \(e\) is a regular edge if and only if \(e\) satisfies Property \(P.\)
\end{theorem}

\begin{proof}
First assume \(e=\{x,y\}\) is a regular edge of a graph \(G.\) Then by Lemmas~\ref{induced} and \ref{no-3-path}, the edge has Property \(P.\) 

Now assume \(e=\{x,y\}\) satistfies Property \(P\)  and that there exists a minimal vertex cover \(Q\) with \(x,y \in Q.\)  
Note that neither \(x\) nor \(y\) can be a leaf, that is a vertex with a single neighbor, of \(G\) since \(Q\) is minimal. 
Moreover, since \(Q\) is minimal, there must exist an edge \(\{a,x\}\) with \(a \not\in Q\) and an edge \(\{y,b\}\) with \(b \not\in Q.\) 
By Property \(P,\) \(a \neq b\) and \(\{a,b\} \in E(G).\) 
But \(\{a,b\}\) is not covered by \(Q,\) a contradiction. 
Thus no minimal vertex cover can contain both \(x\) and \(y,\) so \(\{x,y\}\) is a regular edge.
\end{proof}

\begin{corollary}\label{leaf-cond}
Let \(G\) be a graph and \(f = \{x,y\} \) an edge of \(G\) with \(y\) a leaf. Then an edge \(e=\{x,z\}\) is a regular edge if and only if \(z\) is a leaf.
\end{corollary}

\begin{proof}
Note that if \(z\) is a leaf, then \(e\) vacuously satisfies Property \(P\) and so is a regular edge.
For the other implication assume that \(z\) is not a leaf. Then there exists an edge \( \{z,w\}\) with \(w \ne x\). 
If \(e\) is a regular edge then by Theorem \ref{good=P}, \(e\) has Property \(P\) and hence there is an edge \(\{y,w\}\) contradicting the assumption that \(y\) is a leaf. 
So, if \(e\) is a regular edge then \(z\) is a leaf.
\end{proof}

Corollary \ref{leaf-cond} has an interesting implication. An edge containing a leaf is always a regular edge since it will satisfy Property $P$ automatically. However, in seeking regular sequences formed by edges, Corollary \ref{leaf-cond} indicates that we need to focus on disjoint edges. Indeed, if $e=\{x,y\}$ is a regular edge of a graph $G$, then $R/(I,x+y)$ is effectively (up to isomorphism) the edge ring of a graph $G'$ with a loop replacing $e$ and $x$ and $y$ identified (see \cite{FHM}). After polarizing the loop, there is an edge $e'=\{x,y'\}$ where $y'$ is a leaf and $N_{G'}(x) = N_G(x) \cup N_G(y).$ Since no edge of $G'$ containing $x$ can be regular, this means no edge of $G$ containing $x$ or $y$ can be regular on (the polarization of) $R/(I,x+y)$, hence the need to focus on disjoint edges to form regular sequences.

\begin{example}
Let $G$ be a star graph with central vertex $x$ and leaves $y_1, \ldots, y_t.$ Then every edge $\{x,y_i\}$ is a regular edge, but since $\depth(R/I(G))=1$, these edges cannot be combined to form a regular sequence of length greater than one. Note that $R/(I,x+y_1) \cong R/J$ where $J$ is the graph with a loop shown below, and $J^{pol} = I$.

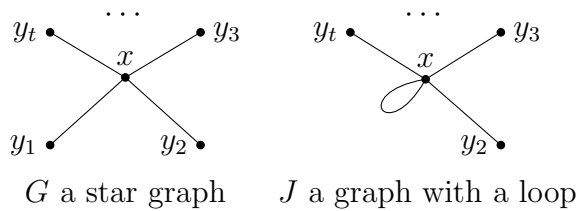
$$
\begin{tabular}{cc}
\begin{tikzpicture}
\tikzstyle{point}=[inner sep=0pt]
\node (x)[point, label=above:{$x$}] at (1,.9) {};
\node (y1)[point, label=left:{$y_1$}] at (0,0) {};
\node (y2)[point, label=left:{$y_2$}] at (2,0) {};
\node (y3)[point, label=right:{$y_3$}] at (2,1.5) {};
\node (yt)[point, label=left:{$y_t$}] at (0,1.5) {};
\node (dots)[point, label=below:{$\cdots$}] at (1,2) {};

\draw[black, fill=black] (x) circle(0.05);
\draw[black, fill=black] (y1) circle(0.05);
\draw[black, fill=black] (y2) circle(0.05);
\draw[black, fill=black] (y3) circle(0.05);
\draw[black, fill=black] (yt) circle(0.05);

\draw (x.center) -- (y1.center);
\draw (x.center) -- (y2.center);
\draw (x.center) -- (y3.center);
\draw (x.center) -- (yt.center);
\end{tikzpicture}

&

\begin{tikzpicture}[every loop/.style={min distance=10mm, in=190, out=240, looseness=10}]
\tikzstyle{point}=[inner sep=0pt]
\node (x)[point, loop below,  label=above:{$x$}] at (1,.9) {};
\node (y2)[point, label=left:{$y_2$}] at (2,0) {};
\node (y3)[point, label=right:{$y_3$}] at (2,1.5) {};
\node (yt)[point, label=left:{$y_t$}] at (0,1.5) {};
\node (dots)[point, label=below:{$\cdots$}] at (1,2) {};

\draw[black, fill=black] (x) circle(0.05);
\draw[black, fill=black] (y2) circle(0.05);
\draw[black, fill=black] (y3) circle(0.05);
\draw[black, fill=black] (yt) circle(0.05);


\draw (x) edge [anchor=center, loop below] (x);
\draw (x.center) -- (y2.center);
\draw (x.center) -- (y3.center);
\draw (x.center) -- (yt.center);
\end{tikzpicture}
\\

$G$ a star graph & $J$ a graph with a loop

\end{tabular}
$$
\end{example}

To avoid this type of situation, we will focus on disjoint edges in later sections when forming regular sequences. Note that when we quotient with the first regular edge, the graph associated to $R/(I,x+y)$ has a loop. Therefore, we first expand our focus to include graphs with loops.

\section{Property $P$ and Graphs with Loops}

The literature surrounding Property $P$ has focused on graphs without loops. In order to be able to apply the property inductively later in this paper, we will need to consider graphs that have loops. To this end, we extend the definition of Property $P$ and results from Section~\ref{regular} to graphs with loops. A graph $G$ has a loop at $x \in V(G)$ if $\{x,x\} \in E(G)$. We write $\{x^2\}\in E(G)$ in this case. We will not consider a vertex with a loop to be a leaf. That is, a leaf is a vertex contained in a single edge that is not a loop.

\begin{definition}\label{Property P loops}
Let \(G\) be a graph, potentially with loops, and \(e=\{x,y\} \in E(G).\)  Then \(e\) is said to have {\em Property P} if  $\{x^2\}, \{y^2\} \not\in E(G)$ and for any \(a,b \in V(G)\) with \(\{a,x\}, \{y,b\} \in E(G)\) we have \(a \neq b\) and \(\{a,b\} \in E(G).\)
\end{definition}

Note that with this definition, the situation with induced subgraphs is similar to that presented in Lemma~\ref{induced}, but the induced subgraph might also contain loops. More precisely, if  \(G\) is a graph potentially with loops and \(e=\{x,y\} \in E(G)\) then \(e\) has Property \(P\) if and only if for every \(a \in N_G(x) \setminus \{y\}\) and \(b \in N_G(y) \setminus \{x\}\) the induced subgraph on \( \{a,x,y,b\} \) is a \(4\)-cycle, potentially with loops at $a$ and $b$. However, using Definition~\ref{Property P loops}, more care needs to be taken to ensure an edge with Property \(P\) is regular. As a first step, we observe that the vertices of a regular edge cannot have loops.

\begin{lemma}\label{regular-no-loop}
Let \(G\) be a graph, potentially with loops, and \(e=\{x,y\}\) an edge of \(G.\) If \(e\) is a regular edge, then \(G\) does not contain a loop at \(x\) or at \(y\). Moreover, at most one of \(N_G(x)\) or \(N_G(y)\) contains vertices with loops.
\end{lemma}

\begin{proof}
First assume \(e=\{x,y\}\) is a regular edge of a graph \(G.\) Since $e$ is regular, for any associated prime $Q$, \(x\) and \(y\) cannot both be in $Q$. If $\{x^2\}$ is an edge of $G$, that is, $G$ has a loop at $x$, then there is an associated prime $Q$ of $R/I(G)$ that contains $N_G[x]$ by \cite[Corollary 4.14]{MV}. Since $x,y \in N_G[x]$, this is a contradiction to $e$ being regular, so $x^2 \not\in E(G)$. Similarly, $y^2 \not\in E(G)$. Now assume there exist $a \in N_G(x)$ and $b \in N_G(y)$ with $a^2, b^2 \in E(G)$. By \cite[Corollary 4.14]{MV}, there exists an associated prime $Q$ of $G$ containing $N_G[a] \cup N_G[b]$. Since $x \in N_G[a]$ and $y\in N_G[b]$, this is a contradiction. Thus at most one of $N_G(x), N_G(y)$ contains loops.
\end{proof}

We now establish that the analogs of Lemmas~\ref{no-triangle} and~\ref{no-3-path} hold for graphs with loops. Note that both Definition~\ref{regular} and Lemma~\ref{regular sums} are based on associated primes and hold for graphs with loops.

\begin{lemma}\label{no-triangle-loops}
If \(G\) is a graph, potential with loops, and \(\{x,y\}\) is a regular edge of \(G,\) then \(\{x,y\}\) is not contained in any triangle in \(G.\)
\end{lemma}

\begin{proof}
Suppose \(b \in V(G)\) is such that \(x,y,b\) form the vertices of a triangle in \(G.\) Then \(x,y \in N_G[b]\).  If \(G\) has a loop at \(b\), then there is an associated prime \(Q\) of \(R/I(G)\) that contains \(N_G[b]\) by \cite[Corollary 4.14]{MV}. Then \(x,y \in Q\), a contradiction.
 If \(G\) does not have a loop at \(b\), then \(W=V(G)\setminus \{b\}\) is a vertex cover of \(G.\) 
Thus there exists \(Q \subseteq W\) such that \(Q\) is a minimal vertex cover of \(G.\) Since \(b \not\in W,\) then \(x\in Q\) because \(\{x,b\}\in E(G)\) and \(y \in Q\) because \(\{y,b\}\in E(G).\)
This contradicts the definition of \(\{x,y\}\) being a regular edge. Thus no such \(b\) exists.
\end{proof}

\begin{lemma}\label{no-3-path-loops}
If \(G\) is a graph, potentially with loops, and \(\{x,y\}\) is a regular edge of \(G,\) then for any \(a,b \in V(G)\) with \(a \in N_G(x)\setminus \{y\}\) and \(b \in N_G(y) \setminus \{x\},\) the induced subgraph of \(G\) on the vertices \(a,x,y,b\) is a \(4\)-cycle with at most one loop. In particular, if \(a \in N_G(x)\) and \(b\in N_G(y),\) then \(\{a,b\}\in E(G).\) 
\end{lemma}

\begin{proof}
First, notice that if \(a = b,\) or  \(\{a,y\}\in E(G),\) or \(\{x,b\}\in E(G)\) then there is a triangle of \( G\) containing \( \{x,y\}\), contradicting Lemma~\ref{no-triangle-loops}. Thus \(a \ne b,\) and  \( \{a,y\}\not\in E(G),\) and \(\{x,b\}\not\in E(G).\) 
By Lemma~\ref{regular-no-loop}, at most one of the vertices in the set \(\{a,x,y,b\}\) contains a loop. If such a loop exists, it is either a loop at \(a\) or at \(b\). 

First assume neither \(a\) nor \(b\) has a loop.
Then if $\{a,b\}\not\in E(G)$, set $W=V\setminus \{a,b\}$. Since $W$ is a vertex cover of $G,$ it contains a minimal vertex cover $Q$. 
By definition, $x \in Q$ since $a \not\in Q$ and $\{x,a\}\in E(G)$ and $y\in Q$ since $b\not\in Q$ and $\{y,b\}\in E(G)$. 
Thus $x,y \in Q$, a contradiction. 

Assume \(G\) has a loop at \(a\) and \(\{a,b\} \not\in E(G).\)  Consider \(W=V(G) \setminus (N_G[a] \cup \{b\})\).  Then \(W\) contains a minimal vertex cover \(K\) of the induced subgraph of \(G\) on \(V(G) \setminus N_G[a].\) Now \(y\in K\) since \(\{b,y\} \in E(G)\) and \(b \not\in K.\)
By \cite[Corollary 4.14]{MV}, there is an associated prime \(Q\) of \(G\) containing \(N_{G}[a] \cup K.\) Since \(x \in N_G[a]\) and \(y\in K\), we have \(x,y \in Q\), a contradiction.
 Hence $\{a,b\} \in E(G)$ and  the induced subgraph of \( G\) on \(a,x,y,b\) is a \(4\)-cycle with at most one loop.
\end{proof}

We are now ready to prove the extension to graphs with loops of the connection between Property \(P\) and regularity.

\begin{theorem} \label{good=P loops}
Let \(G\) be a graph, potentially with loops, and \(e=\{x,y\}\) an edge of \(G.\) Then \(e\) is a regular edge if and only if \(e\) satisfies Property \(P\) and at most one of \(N_G(x)\) or \(N_G(y)\) contains vertices with loops.
\end{theorem}

\begin{proof}
First assume \(e=\{x,y\}\) is a regular edge of a graph \(G.\) By Lemma~\ref{regular-no-loop}, \(G\) does not contain a loop at \(x\) or \(y\) and at most one of \(N_G(x), N_G(y)\) contains loops. If \(a,b \in V(G)\) with \(\{a,x\},\{y,b\} \in E(G)\) then by Lemma~\ref{no-triangle-loops}, \(a \ne b\) and by Lemma~\ref{no-3-path-loops}, \(\{a,b\}\in E(G).\)Thus the conditions for Property \(P\) in Definition~\ref{Property P loops} are satisfied.

Now assume \(e=\{x,y\}\) satistfies Property \(P\)  and and at most one of $N(x)$ or $N(y)$ contains vertices with loops. As in the proof of Theorem~\ref{good=P}, 
assume that there exists a minimal  vertex cover \(Q\) with \(x,y \in Q.\)  
Note that neither \(x\) nor \(y\) can be a leaf of \(G\) since \(Q\) is minimal. 
Moreover, since \(Q\) is minimal, there must exist an edge \(\{a,x\}\) with \(a \not\in Q\) and an edge \(\{y,b\}\) with \(b \not\in Q.\) 
By Property \(P,\) \(a \neq b\) and \(\{a,b\} \in E(G).\) 
But \(\{a,b\}\) is not covered by \(Q,\) a contradiction. 
Thus no minimal vertex cover can contain both \(x\) and \(y.\) 

By \cite[Corollary 4.14]{MV}, the associated primes of \(R/I(G)\) that are not minimal are constructed from (unions of) closed neighborhoods of loops extended to vertex covers of \(G\). By hypothesis, \(x\) and \(y\) are not both contained in the union of all closed neighborhoods of vertices with loops. Since \(x\) and \(y\) are not both contained in any minimal vertex cover, if \(x\) and \(y\) are both contained in an embedded associated prime \(Q\), then one of them, say \(x,\) is in the closed neighborhood of a loop \(a\) and the other, say \(y,\) is necessary to extend the the set formed from closed neighborhoods of loops to contain a vertex cover. That means there is an edge \(\{y,b\}\) with \(b \not\in Q\). By Property \(P\), \(\{a,b\}\in E(G)\) and \(b \in N[a] \subseteq Q\), a contradiction. 
Thus no associated prime of \(R/I(G)\) contains both \(x\) and \(y\), so \(\{x,y\}\) is a regular edge.
\end{proof}

We conclude this section with the analog of Corollary~\ref{leaf-cond} for graphs with loops. As before, this corollary is a motivating factor in focusing on distjoint edges.

\begin{corollary}\label{leaf-cond loops}
Let \(G\) be a graph and \(f = \{x,y\} \) an edge of \(G\) with \(y\) a leaf. Then an edge \(e=\{x,z\}\) is a regular edge if and only if \(z\) is a leaf and $x^2 \not\in E(G)$.
\end{corollary}

\begin{proof}
Note that if \(z\) is a leaf and $x^2 \not\in E(G)$, then \(e\)  satisfies Property \(P\). Since $z$ is a leaf, the conditions of Theorem~\ref{good=P loops} are satisfied and so $e$ is a regular edge.
For the other implication assume that \(z\) is not a leaf. Then there exists an edge \( \{z,w\}\) with \(w \ne x\). 
If \(e\) is a regular edge then by Theorem \ref{good=P loops}, \(e\) has Property \(P\) and hence there is an edge \(\{y,w\}\) contradicting the assumption that \(y\) is a leaf. 
So, if \(e\) is a regular edge then \(z\) is a leaf.
\end{proof}

\section{Matchings and Regular sequences of edges}\label{Regular Sequences}

The goal of this section is to form a special type of regular sequence using edges that satisfy Property $P$. In general, if $x+y$ is a regular element on $R/I(G)$ for a graph $G$, then $R/(I(G),x+y) \cong R/(J, x+y)$ where $J$ is the ideal formed from $I(G)$ by replacing $y$ by $x$ in each generator divisible by $y$ (see \cite[Theorem 2.6]{FHM}). Since the ideal $J$ is not square-free, indeed, it is the edge ideal of a graph with a loop at vertex $x$, we will define two new graphs that will allow us to toggle between working modulo the regular element and working with the loop-free graph corresponding to a polarization of $J$.

\begin{definition}\label{Ge}
Let $G$ be a graph, potentially with loops, and fix an edge $e=\{x,y\} \in E(G)$ with $x \neq y$. Define $G_e$ to be the graph with $V(G_e) = V(G)\setminus \{y\}$ and edges defined by:
\begin{itemize}
\item[(i)] $\{a,b\} \in E(G_e)$ if $\{a,b\} \in E(G)$ and $a,b \neq y$, or
\item[(ii)] $\{a,x\}\in E(G_e)$ if $\{a,y\} \in E(G)$.
\end{itemize}
\end{definition}

Note that since $\{x, y\}$ is an edge of $G$, the second criteria above gives that $\{x, x\}$ is an edge of $E(G_e)$, that is, $G_e$ has a loop at $x$. In addition, using the notation from the start of this section, $J = I(G_e)$.

\begin{definition}\label{polarizedGe}
Let $G$ be a graph, potentially with loops, and fix an edge $e=\{x,y\} \in E(G)$ with $x \neq y$. Define $G^e$ to be the graph with $V(G^e) = V(G)$ and edges defined by:
\begin{itemize}
\item[(i)] $\{a,b\} \in E(G^e)$ if $\{a,b\} \in E(G)$ and $a,b \neq y$, or
\item[(ii)] $\{a,x\}\in E(G^e)$ if $\{a,y\} \in E(G)$ and $a \ne x$, or
\item[(iii)] $\{x,y\}\in E(G^e)$.
\end{itemize}
\end{definition}

Note that if $G$ does not have loops, then $I(G^e) = (I(G_e))^{pol}$. In general, $I(G^e)$ is formed from $I(G_e)$ by polarizing only the loop created from $e$. For convenience, we will denote $(G_{e_1})_{e_2}$ as $G_{e_1,e_2}$ and $(G^{e_1})^{e_2}$ as $G^{e_1,e_2}$ in the remainder of the paper.

The next lemma shows how Property $P$ for an edge $f$ passes to $G_e$ and $G^e$ when $e$ and $f$ are disjoint edges. Note that in order to iterate the process of passing Property \(P\) through the contraction of multiple edges, the hypotheses of the following results allow for the graph $G$ to have loops.

\begin{lemma}\label{P descends}
Let $G$ be a graph, potentially with loops. Suppose $\{u,v\} \in E(G)$ has Property $P$ and $e=\{x,y\}$ is an edge of $G$ with $u,v,x,y$ distinct and the induced graph on $u,v,x,y$ is not a $4$-cycle, up to potential loops on \(x\) and \(y\). Then the image of $\{u,v\}$ has Property $P$ in both $G_e$ and $G^e$. 
\end{lemma}

\begin{proof}
Note that since $u,v,x,y$ are distinct, the image of $\{u,v\}$ in both $G_e$ and $G^e$ is again $\{u,v\}$ by condition $(i)$ of Definitions~\ref{Ge} and~\ref{polarizedGe}. Also by definition, since the vertices are distinct and $\{u,v\}$ has Property $P$ in $G$, then $\{u^2\}, \{v^2\}$ are not in $E(G_e)$ or $E(G^e)$. 

Now let $a,b \in V(G_e)$ with $\{a,u\}, \{v,b\} \in E(G_e)$. If neither $a$ nor $b$ is equal to $x$, then $\{a,u\}, \{v,b\} \in E(G)$ and thus $a \neq b$ and $\{a,b\}\in E(G)$ and so by Definition~\ref{Ge} $(i)$ $\{a,b\} \in E(G_e)$.

 If $a=x$ and $b=x$, then since in $G$ neither $x$ nor $y$ is a common neighbor of $u$ and $v$ since $\{u,v\}$ has Property $P$, precisely one of $\{a,u\}$ or $\{v,b\}$ was produced by situation $(ii)$ of Definition~\ref{Ge}. Without loss of generality, assume $\{v,b\} = \{v,x\}$ was produced by the edge $\{v,y\}$ in $E(G)$. Then the induced graph in $G$ on $u,v,x,y$ contains a $4-$cycle. Note that since $\{u,v\}$ has Property $P$ in $G$, $u,v$ do not have a common neighbor, meaning that the $4-$cycle is induced modulo potential loops on \(x\) and \(y\), a contradiction. 

Finally assume $a\neq x$ and $b = x$. Then either $\{v,x\}$ or $\{v,y\}$ is an edge of $G$. Then either $\{a,x\}$ or $\{a,y\}$ respectively is in $E(G)$ since $\{u,v\}$ has Property $P$. In either case, $\{a,x\}$ is an edge of $G_e$ and thus $\{u,v\}$ has Property $P$ in $G_e$.

An analgous argument holds for $G^e$.
\end{proof}

The following is the key lemma in considering regular sequences of edges. For convenience, set \(L=\{v\in V(G) \mid v^2 \in E(G)\}\) to be the set of loops in \(G\).

\begin{lemma}\label{two edges}
Suppose $\{x_1,y_1\}$ and $\{x_2,y_2\}$ are disjoint regular edges of a graph $G$, potentially with loops. Then  $\{x_2,y_2\}$ is a regular edge of $H=G_{\{x_1,y_1\}}$ if and only if 
\begin{itemize}
\item  if $N_G(x_2) \cap \{x_1,y_1\} \ne \emptyset$ then $N_G(y_2) \cap (\{x_1, y_1\} \cup L)=\emptyset$, and
\item  if $N_G(y_2) \cap \{x_1,y_1\} \ne \emptyset$ then $N_G(x_2) \cap (\{x_1, y_1\} \cup L)=\emptyset$.
\end{itemize}
Moreover, $x_1+y_1, x_2+y_2$ is a regular sequence on $R/I(G)$.
\end{lemma}

\begin{proof}
 Set $e=\{x_1,y_1\}$. We need to show the image of $\{x_2,y_2\}$ is a regular edge of $H=G_e$. By Theorem~\ref{good=P loops}, \(\{x_2,y_2\}\) has Property \(P\) in \(G\).
 By Definition~\ref{Property P loops}, neither $x_2$ nor $y_2$ has a loop in $G$. By definition, there is only one new loop in $G_e$, which is located at $x_1$. Thus in $H$, neither $x_2$ nor $y_2$ has a loop.  By Theorem~\ref{good=P loops} it is enough to show that $\{x_2,y_2\}$ satisfies Property $P$ in the graph $H$ and at most one of $N_H(x_2)$ and $N_H(y_2)$ contains a loop. By hypothesis, at most one of these neighbor sets contains $x_1$ and the other set does not contain a loop in $G$ or in $H$.  Note that the induced graph in \(G\) on \(\{x_1,y_1,x_2,y_2\}\) is not a $4$-cycle by the conditions imposed on \(N_G(x_2)\) and \(N_G(y_2)\). By regularity, there are no loops on any of these $4$ vertices. Thus by Lemma~\ref{P descends}, \(\{x_2,y_2\}\) has Property \(P\) in \(H\).

For the converse, assume \(\{x_2,y_2\}\) is a regular edge of \(H\) and  \(N_G(x_2) \cap \{x_1,y_1\} \ne \emptyset\). Then \(\{x_1,x_2\} \in E(H)\). Recall that there is a loop at \(x_1\) in \(H\). Suppose  \(N_G(y_2) \cap (\{x_1, y_1\} \cup L) \ne \emptyset\). Then there is a vertex \(a \in \{x_1, y_1\} \cup L\) with \(\{a,y_2\} \in E(G).\) If \(a=x_1\) or \(a=y_1\), then \(\{x_1,y_2\}\in E(H)\). Otherwise, \(\{a,y_2\}\in E(H).\) In either case, in \(H\), both \(x_2\) and \(y_2\) are adjacent to loops, a contradiction to Theorem~\ref{good=P loops} since \(\{x_2,y_2\}\) is a regular edge of \(H\).

Finally, by Lemma~\ref{regular sums} and the comment following Lemma~\ref{regular-no-loop}, $x_1+y_1$ is regular on $R/I(G)$. Now 
$$R/(I(G),x_1+y_1) \cong R/((I(H), x_1+y_1) \cong R'/I(H)$$
where $R'$ is a polynomial ring in one fewer variable (see the proof of \cite[Theorem 2.6]{FHM}). More specifically, the variables of $R'$ are $V(H)=V(G_{\{x_1,y_1\}})=V(G)\setminus\{y_1\}$. Since $\{x_2,y_2\}$ is a regular edge of $H$, the final statement follows.
\end{proof}

Note that the edges being disjoint in Lemma~\ref{two edges} is important. If we start with two distinct edges $\{x_1,y_1\}$ and $\{x_1,y_2\}$ that are each regular in \(G\) but which share a vertex, then $\{x_1,y_2\}$ will not be regular in $H=G_{\{x_1,y_1\}}$, since there is a loop in \(H\) at \(x_1\).

\begin{remark}
The proof of Lemma~\ref{two edges} can be easily adapted to show that under the hypotheses, \(\{x_2,y_2\}\) is a regular edge of \(G^{\{x_1,y_1\}}\) as well. However the converse no longer holds, as illustrated by the edges \(\{x_1,y_1\}, \{x_2,y_2\}\) in the graph
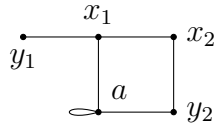
\begin{center}
\begin{tikzpicture}[every loop/.style={}]
\tikzstyle{point}=[inner sep=0pt]
\node (a)[point,label=above: $x_1$] at (1,1) {};
\node (b)[point,label=right:$x_2$] at (2,1){};
\node (c)[point,label=right:$y_2$] at (2,0){};
\node (d)[point,label=above right:$a$] at (1,0){};
\node (e)[point,label=below:$y_1$] at (0,1){};

\draw (a.center) -- (b.center);
\draw (b.center) -- (c.center);
\draw (c.center) -- (d.center);
\draw (d.center) -- (a.center);
\draw (a.center) -- (e.center);

\draw (d) edge [anchor=center, loop left] (d);

	\filldraw [black] (a.center) circle (1pt);
	\filldraw [black] (b.center) circle (1pt);
	\filldraw [black] (c.center) circle (1pt);
	\filldraw [black] (d.center) circle (1pt);
	\filldraw [black] (e.center) circle (1pt);
\end{tikzpicture}
\end{center}
Note that \(\{x_2,y_2\}\) is a regular edge of \(G^{\{x_1,y_1\}}\) but both \(N_G(x_2) \cap \{x_1,y_1\}\ne \emptyset\) and \(N_G(y_2) \cap \{x_1, y_1, a\} \ne \emptyset.\)
\end{remark}

\begin{corollary}\label{reg 2}
Let $G$ be a graph without loops and assume $\{x,y\}, \{u,w\}$ are disjoint regular edges of $G$ such that $\{x,y,u,w\}$ are not the vertices of a square. Then $x+y, u+w$ is a regular sequence on $R/I(G)$.
\end{corollary}

\begin{proof}
Note that $I(G_{\{x,y\}})$ is the edge ideal of a graph with precisely one loop, $\{x,x\}$. The result follow immediately from Lemma~\ref{two edges}.
\end{proof}

We next establish conditions that will allow us to build longer regular sequences from disjoint regular edges. Although stated in a slightly different format, when $t=2$ the conditions of the theorem below are equivalent to the hypotheses of Lemma~\ref{two edges} since at most one vertex of a regular edge of $G$ has a neighbor in $G$ that is a loop.

\begin{theorem}\label{reg seq}
Let \(G\) be a graph, potentially with loops on vertices \(\{z_1, \ldots, z_r\}\). Assume \(\{x_1,y_1\}, \ldots ,\{x_t,y_t\}\) are disjoint regular edges of \(G\) such that for \(2 \leq i \leq t\) either 
\[ N_G(x_i) \cap \{x_1,y_1,\ldots , x_{i-1},y_{i-1}, z_1, \ldots, z_r\} = \emptyset, \,\, \mbox{\rm{or}}\]
\[N_G(y_i) \cap \{x_1,y_1,\ldots , x_{i-1},y_{i-1},z_1, \ldots, z_r\} = \emptyset.\]
 Then $x_1+y_1, x_2+y_2,\ldots,x_t+y_t$ is a regular sequence on $R/I(G)$.
\end{theorem}

\begin{proof}
We proceed by induction on \(t\), with Lemma~\ref{regular sums} and the comment following Lemma~\ref{regular-no-loop} establishing the result for \(t=1\) and Lemma~\ref{two edges} establishing the result for \(t=2\). Assume \(t \geq 3\) and that the result holds for \(t-1\).
Consider 
\[R/(I(G), x_1+y_1, \ldots , x_{t-1}+y_{t-1}) \cong R''/I(H)\]
where $H = G_{\{x_1,y_y\},\{x_2,y_2\}, \ldots ,\{x_{t-1},y_{t-1}\}}$ and $R'' \cong R/(y_1, \ldots, y_{t-1})$. We will show that $\{x_t, y_t\}$ is a regular edge of $H$.

First, by hypothesis, for $2 \leq s \leq t$, we have that $\{x_1,y_1\},\{x_s,y_s\}$ are disjoint regular edges of \(G\) and if \(N_G(x_s) \cap \{x_1,y_1\} \ne \emptyset\), then 
\[N_G(y_s) \cap \{x_1,y_1,\ldots , x_{s-1},y_{s-1},z_1, \ldots, z_r\} = \emptyset,\]
so in particular 
\(N_G(y_s) \cap \{x_1, y_1, z_1, \ldots, z_r\}=\emptyset.\) A similar argument holds when \(N_G(y_s) \cap \{x_1,y_1\} \ne \emptyset\). Thus by Lemma~\ref{two edges}, 
 \(\{x_s,y_s\}\) is regular in \(G_1 = G_{\{x_1,y_1\}}\) for all \(2 \leq s \leq t\). 

Inductively, define \(G_i = G_{\{x_1,y_1\},\{x_2,y_2\}, \ldots ,\{x_{i},y_{i}\}}\) for \(i \leq t-1\). Notice that \(G_i\) is a graph with loops at \(x_1, \ldots , x_{i},z_1,\ldots,z_r\)  and no other loops.
For \(i+2 \le s \le t\),  \(\{x_{i+1},y_{i+1}\}\), \(\{x_s,y_s\}\) are disjoint edges of \(G_i\) that are regular in \(G_i\) by induction on \(i\).
If \(N_{G_i}(x_s) \cap \{x_{i+1},y_{i+1}\} \ne \emptyset\), then 
\[N_{G_{i}}(y_s) \cap \{x_{i+1},y_{i+1},\ldots , x_{s-1},y_{s-1},x_1,\ldots, x_i, z_1, \ldots, z_r\} = \emptyset,\]
so in particular 
\(N_{G_i}(y_s) \cap \{x_{i+1}, y_{i+1},x_1,\ldots,x_i, z_1, \ldots, z_r\}=\emptyset.\) A similar argument holds when \(N_{G_i}(y_s) \cap \{x_{i+1},y_{i+1}\} \ne \emptyset\). Thus by Lemma~\ref{two edges}, 
 \(\{x_s,y_s\}\) is regular in \(G_{i+1}.\) 
 In particular, \(\{x_t,y_t\}\) is regular in \(G_{t-1} = H\) as desired. Hence $x_t+y_t$ is regular on $R''/I(H)$ and the result folllows.
\end{proof}

\begin{remark}
If \(G\) is a graph without loops, the results of Theorem~\ref{reg seq} imply that if \(G\) has a perfect matching with Property \(P\), then these edges form a regular sequence, in which case \(G\) is Cohen-Macaulay, as long as they can be ordered so that each edge has (at least) one vertex not connected to the prior edges. Such will be the case for well-known classes of graphs, such as suspensions of graphs and Cohen-Macaulay bipartite graphs. In addition, by exchanging vertex labels of \(x_i,y_i\) if necessary, one can assume that \(x_i\) is independent from the set \(\{x_1,y_1,\ldots,x_{i-1},y_{i-1}\}\) in this case, or that \[N_G(x_i) \cap \{x_1,y_1,\ldots , x_{i-1},y_{i-1}, z_1, \ldots, z_r\} = \emptyset\] in the more general setting of Theorem~\ref{reg seq}.
\end{remark}

In the special case where $G$ is a Cohen-Macaulay graph and the edges used in Theorem~\ref{reg seq} form a perfect matching, then the precise form of $G$ is known. 

\begin{remark} If $\{x_1,y_1\}, \ldots, \{x_t,y_t\}$ in Theorem~\ref{reg seq} forms a perfect matching, then $G$ is a very well-covered graph, and the results follow from \cite{Crupi2011}. To see this, note that by relabeling if necessary, we can assume $N(x_i) \cap \{x_1,y_1,\ldots , x_{i-1},y_{i-1}\} = \emptyset$ for all $i$. Thus $\{x_1, \ldots, x_t\}$ is an independent set. Indeed, if $\{x_i,x_j\} \in E(G)$ with $i<j$ then $N(x_j) \cap \{x_1,y_1,\ldots , x_{j-1},y_{j-1}\} \ne \emptyset$. Then $\{y_1, \dots ,y_t\}$ is a minimal vertex cover and since the matching is perfect, $t = 1/2|V(G)| = \height(I)$.
\end{remark}

\section{Linear binomial regular elements}

Theorem~\ref{good=P} establishes a means of creating linear binomials that are regular on \(R/I(G)\) for any (simple) graph \(G\) conaining an edge with Property \(P\). Theorem~\ref{good=P loops} does the same for graphs with loops. In this section, we show that the only linear binomials that are regular on \(R/I(G)\) correspond to either edges with Property \(P\) in \(G\) or pairs of non-connected vertices \(x,y\) for which adding an edge \(\{x,y\}\) to \(G\) results in the new edge having Property \(P\) in the augmented (simple) graph. An analgous statement holds for a graph with loops, assuming conditions similar to those in previous sections.

\begin{theorem}\label{binomial regular}
Let \(G\) be a graph with \(x,y \in V(G)\) but \(\{x,y\}\not\in E(G).\)
Then the following are equivalent
\begin{enumerate}
\item In the graph \(\widehat{G}\) with \(V(\widehat{G}) = V(G)\) and \(E(\widehat{G}) = E(G)\cup \{\{x,y\}\},\) the edge \(\{x,y\}\) has Property \(P\).
\item The elements \(x+y\) and \(x-y\) are regular elements of \(R/I(G)\). 
\item The elements \(x+y\) and \(x-y\) are regular elements of \(R/I(\widehat{G})\). 
\item For any associated prime \(Q\) of \(R/I(G)\) the  element \(xy\) is not an element of \(Q^2.\)
\item For any associated prime \(Q\) of \(R/I(\widehat{G})\) the  element \(xy\) is not an element of \(Q^2.\)
\item No minimal vertex cover of \(G\)  contains the set \(\{x,y\}.\)
\item No minimal vertex cover of \(\widehat{G}\)  contains the set \(\{x,y\}.\)
\end{enumerate}
\end{theorem}

\begin{proof}
As in  Lemma~\ref{regular sums}, \(x+y\) and \(x-y\) are regular on \(R/I(G)\) if and only if no associated prime of \(I(G)\), or equivalently, no minimal vertex cover of \(G\), contains both \(x\) and \(y\). This shows the equivalence of \((2),(4),\) and \((6)\). By Theorem~\ref{good=P}, Lemma~\ref{regular sums}, and Definition~\ref{regular}, \((1),(3),(5)\) and \((7)\) are equivalent. Thus it suffices to show \((6)\) and \((7)\) are equivalent.

Suppose that no minimal vertex cover of \(G\)  contains the set \(\{x,y\}.\)   Let \(X\) be a minimal vertex cover of \(\widehat{G}\) that contains the set \(\{x,y\}\).
Then \(X\) is a vertex cover of \(G\) and so contains a minimal vertex cover \(Y\subset X\) of \(G.\)  By the assumption that no minimal vertex cover of \(G\) contains the set \(\{x,y\},\)
at most one of the vertices \(x,y\) is in the set \(Y\). Without loss of generality, assume \(x \not\in Y\). Then \(Y\cup\{y\}\) is a proper subset of \(X\) that is a cover for \(\widehat{G}\) contradicting the minimality of \(X.\) Thus \((6) \Rightarrow (7).\)

Suppose that in the graph \(\widehat{G}\) , no minimal vertex cover of \(\widehat{G}\) contains the set  \(\{x,y\}.\)
Let \(Z\) be a minimal vertex cover of \(G\) that contains \(\{x,y\}.\)  Then \(Z\) is a vertex cover of \(\widehat{G}\) that contains a minimal vertex cover \(Y\) of \(\widehat{G}.\)
This is a contradiction, as \(Y\) is a vertex cover of \(\widehat{G}\), and so is a vertex cover of \(G\), but \(Y\) is a proper subcover of \(Z,\) contradicting the minimality of \(Z.\) Hence no minimal vertex cover of \(G\) contains the set \(\{x,y\}.\) Thus \((7) \Rightarrow (6).\)

\end{proof}

Note that Theorem~\ref{binomial regular} recovers the regularity of the sum of vertices forming a "leaf pair" as in \cite[Theorem 4.11]{FHM}. Since it classifies all binomial regular elements, it can be combined with \cite[Proposition 4.1]{FHM} to provide lower bounds on depths of edge ideals of graphs.

\section{Application to Hilbert Series of Edge Ideals of Graphs}

The regular sequences found in Section~\ref{Regular Sequences} consist of homogeneous forms of degree one. This allows one to use these regular sequences to simplify the computation of the Hilbert Series of $R/I$. In this section, we show how such an approach leads to a simplification of the standard correspondence between an $h$-vector of an ideal and the $f$-vector of an associated simplicial complex. Moreover, in the Cohen-Macaulay case, when there is a perfect matching formed by a regular sequence of edges satisfying Property $P$ with the necessary condition on the neighbor sets in Theorem~\ref{reg seq}, then we provide an interpretation of the Hilbert coefficients appearing in the $h$-vector.

First recall that when $I$ is a square-free monomial ideal, the {\em Hilbert Series} of $R/I$ can be written as
$$ HS_{R/I} = \sum_{s=0}^{\infty} \dim_k(R/I)_s t^s = \frac{h_0+h_1t+\cdots +h_dt^d}{(1-t)^d}$$
where $d = \dim(R/I).$ For ease of reference, we collect the coefficients of the numerator of this series in the {\em $h$-vector} of $R/I$, which is $(h_0,h_1,\ldots, h_d).$

Given an edge ideal $I(G)$, consider the simplicial complex $I_{\Delta}$ formed by the clique complex of the complement graph of $G$. That is, $\sigma =  \{x_{i_1}
, \ldots,x_{i_s}\}$ is a face of $I_{\Delta}$ if and only if $\sigma$ is an independent set. The {\em $f$-vector of $I_{\Delta}$} is the vector $(f_{-1}, f_0, f_1, \ldots, f_{d-1})$ where $f_{-1}=1$ represents the empty subset, $f_0=n$ is the number of vertices in $G$, and in general $f_i$ is the number of subsets of $V(G)$ containing $i+1$ independent vertices, which is the number of faces of dimension $i$ of $I_{\Delta}$. 
There is a well-known relationship between the $h$-vector of $I$ and the $f$-vector of the simplicial complex whose Stanley-Reisner ideal is $I$. While this relationship is well-known and holds for any square-free monomial ideal, we state it here for graphs for ease of reference.

\begin{theorem}\cite{Stanley}
Assume $I$ is an edge ideal of a graph $G$ and $(h_0, \ldots, h_d)$ is the $h$-vector of $I$. If $(f_{-1},f_0, \ldots, f_{d-1})$ is the $f$-vector of $I_{\Delta}$. Then

$$h_k = \sum_{i=0}^k (-1)^{k-i} \binom{d-i}{k-i}f_{i-1}$$
\end{theorem}

To better understand $h_i$, we define a simplicial complex that is smaller than $I_{\Delta}$ and relate $h_i$ to the $f$-vector of this smaller complex. 

\begin{lemma}\label{HS equal}
If $G$ is a graph, potentially with loops, and $e=\{x,y\}$ is a regular edge of $G$, then $G$ and $G^e$ have identical Hilbert Series. Moreover, the $h$-vectors of $I(G), I(G_e),$ and $I(G^e)$ are all equal.
\end{lemma}

\begin{proof}
Since $x+y$ is a homogeneous regular element on $R/I$ of degree $1$, then 
$$HS_{R/(I,x+y)} = (1-t) HS_{R/I}.$$
Recall that 
$$R/(I,x+y) = R/(I(G_e),x+y) \cong R'/I(G_e)$$
where $R'$ is a ring in one fewer variable. Thus
$$HS_{R'/(I(G_e))} = (1-t) HS_{R/I}$$
and so the $h$-vectors of $I=I(G)$ and of $I(G_e)$ are the same.

Similarly, since $G^e$ is the polarization of $G_e$, we have 
$$R/(I(G^e),x+y) = R/(I(G_e),x+y) \cong R'/I(G_e)$$
and so 
$$HS_{R'/(I(G_e))} = (1-t) HS_{R/I(G^e)}.$$
Hence $I(G_e)$ has the same $h$-vector as $I(G^e)$. Combining the above equalities yields $HS_{R/I} = HS_{R/I(G^e)}$ as well.
\end{proof}

This process can be iterated using a longer regular sequence.

\begin{proposition}\label{HS}
Let \(G\) be a graph, potentially with loops at \(z_1, \ldots, z_r,\) and assume $\{x_1,y_1\}, \ldots ,\{x_t,y_t\}$ are disjoint regular edges of \(G\) such that for \(2 \leq i \leq t\) either 
\[N(x_i) \cap \{x_1,y_1,\ldots , x_{i-1},y_{i-1},z_1, \ldots, z_r\} = \emptyset, \,\, \mbox{\rm{or}}\]
\[N(y_i) \cap \{x_1,y_1,\ldots , x_{i-1},y_{i-1},z_1, \ldots, z_r\} = \emptyset.\] 
Then \(I(G_{e_1, \ldots, e_t})\) and \(I(G^{e_1, \ldots ,e_t})\) both have the same \(h\)-vector as \(I=I(G)\) and the Hilbert series of \(R/I\) is equal to the Hilbert Series of \(R/(I(G^{e_1, \ldots, e_t}))\).
\end{proposition}

\begin{proof}
By Theorem~\ref{reg seq}, $x_1+y_1, x_2+y_2,\ldots,x_t+y_t$ is a regular sequence on $R/I(G)$. By repeated use of Lemma~\ref{HS equal}, the result follows.
\end{proof}

The corollary above allows one to essentially replace edges in a graph that have the Property $P$ with whiskers, which are edges with one vertex a leaf, as long as the condition on the neighbor sets is satisfied. This allows the apply the results of \cite{Brennan-Morey} on whiskered graphs to a more general class of graphs.

\begin{corollary}\label{CM}
If $I=I(G)$ is Cohen-Macaulay and $e_1=\{x_1,y_1\}, \ldots ,e_d=\{x_d,y_d\}$ is a perfect matching of distinct regular edges such that for $2 \leq i \leq t$ either $N(x_i) \cap \{x_1,y_1,\ldots , x_{i-1},y_{i-1}\} = \emptyset$ or $N(y_i) \cap \{x_1,y_1,\ldots , x_{i-1},y_{i-1}\} = \emptyset$, then the $h$-vector of $I$ is the $f$-vector of the simplicial complex corresponding to the induced graph on $\{x_1, \ldots, x_d\}$ in the whiskered graph $G^{e_1, \ldots,e_d}$.
\end{corollary}

\begin{proof}
This follows immediately from Proposition~\ref{HS} and \cite[Corollary 3.6]{Brennan-Morey}.
\end{proof}

The power of Corollary~\ref{CM} is that it directly interprets the entries of the $h$-vector as the entries of an $f$-vector of a smaller simplicial complex than $I_{\Delta}$. Moreover, there is a direct interpretation of the entries of this $f$-vector in terms of the edges of the perfect matching.This leads to a main result of the paper. As mentioned above, after potentially relabeling, one can assume the condition on the neighbor sets holds for \(x_i\).

\begin{theorem}\label{main}
Let $G$ be a graph without loops and let \(M=\{\{x_1,y_1\}, \ldots,\{x_t,y_t\}\}\) be a perfect matching of \(G\) with Property \(P\) such that for \(2 \leq i \leq t, \,\, N(x_i) \cap \{x_1,y_1,\ldots , x_{i-1},y_{i-1}\} = \emptyset.\) Then \(h_i\) is the number of independent subsets of \(M\) of size $i$.
\end{theorem}

\begin{proof}
Set $e_i = \{x_i,y_i\}$. Since $M$ is a perfect matching with Property \(P\), \(G\) is Cohen-Macualay, so by Corollary~\ref{CM}, the \(h_i = f_{i-1}\) where \(f=(f_{-1}, \ldots ,f_{d-1})\) is the \(f\)-vector of the simplicial complex corresponding to induced subgraph of $G^{e_1, \ldots, e_d}$ on the vertices $x_1, \ldots, x_d.$ Then $f_{i-1}$ is the number of  subsets of size $i$ of $\{x_1, \ldots, x_d\}$ that are independent in $G^{e_1, \ldots, e_d}$. By the definition of $G^{e_1, \ldots,e_d}$, a set $\{x_{j_1}, \ldots, x_{j_i}\}$ is independent in $G^{e_1, \ldots,e_d}$ if and only if $e_{j_1}, \ldots, e_{j_i}$ is an induced matching of $G$.
\end{proof}

 This leads to an interesting result on the degree of the numerator of the Hilbert series in this setting in terms of the induced matching number of $G$, which we will denote $ind(G)$.

\begin{corollary}
Assume $I=I(G)$ is Cohen-Macaulay and $e_1=\{x_1,y_1\}, \ldots ,e_d=\{x_d,y_d\}$ is a perfect matching of distinct regular edges such that for $2 \leq i \leq t$ either $N(x_i) \cap \{x_1,y_1,\ldots , x_{i-1},y_{i-1}\} = \emptyset$ or $N(y_i) \cap \{x_1,y_1,\ldots , x_{i-1},y_{i-1}\} = \emptyset$. If $s = \max\{i \mid h_i \ne 0\}$, then $s \leq ind(G)$.
\end{corollary}

\begin{proof}
This follows immediately from Theorem~\ref{main} since in this setting $h_s$ is the number of subsets of $e_1, \ldots, e_d$ of size $s$ that form an induced matching. Since $h_s \ge 0$, we have $ind(G) \geq s$.
\end{proof}

Proposition~\ref{HS} can also be applied in the more general setting where the regular sequence does not form a perfect matching. The result will be a partially whiskered graph with the same Hilbert series, allowing the use of techniques in \cite[Section 4]{Brennan-Morey}. 

We conclude the paper with some illustrative examples.

\begin{example}
If \(G\) is a Cohen-Macaulay bipartite graph, then by \cite{H-H} \(G\) has a perfect matching \(\{x_1,y_1\}, \ldots, \{x_d,y_d\}\) such that if \(\{x_i,y_j\}\in E(G)\), then \(i\leq j\) and if \(\{x_i,y_j\},\{x_j,y_k\} \in E(G)\) then \(\{x_i,y_k\} \in E(G)\). The second condition implies that each edge of the matching has Property \(P\) and the first condition implies \(N_G(x_i) \cap \{x_1, y_1, \ldots, x_{i-1},y_{i-1}\} = \emptyset\) and so Theorem~\ref{reg seq} applies. Applying Corollary~\ref{CM}, one can then compute the $h$-vector directly from the $f$-vector of an appropriate simplicial complex, formed using the clique complex of the complement of the connectivity graph on the perfect matching. Equivalently, $h_i$ is the number of subsets of cardinality $i$ of the matching for which no two edges of the subset are connected in \(G\).

 For a specific example, consider the Cohen-Macaulay bipartite graph on \(8\) vertices depicted below. The center graph represents \(G_{\{x_1,y_1\}\{x_2,y_2\}\{x_3,y_3\}\{x_4,y_4\}}\) and the right hand graph depicts \(G^{\{x_1,y_1\}\{x_2,y_2\}\{x_3,y_3\}\{x_4y_y\}}\), all of which have the same   \(h\)-vector, \((1,4,2)\). Note  the non-loop edges of the center graph encode connections between the edges of the perfect matching, while independent sets of these vertices (ignoring the loops) determine the $h$-vector.

\vspace{-.5in}
\noindent\begin{minipage}{\textwidth}
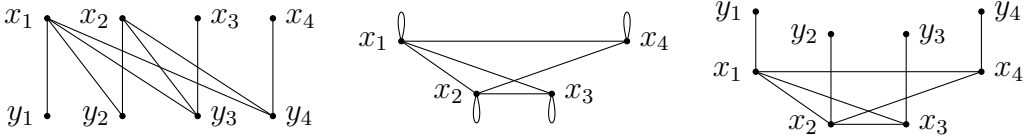
\begin{minipage}[c][6cm][c]{0.3\textwidth}
\begin{tikzpicture}
\tikzstyle{point}=[inner sep=0pt]
\node (a)[point,label=left: $x_1$] at (-1,1.3) {};
\node (b)[point,label=left:$x_2$] at (0,1.3){};
\node (c)[point,label=right:$x_3$] at (1,1.3){};
\node (d)[point,label=right:$x_4$] at (2,1.3){};
\node (e)[point,label=left:$y_1$] at (-1,0){};
\node (f)[point,label=left:$y_2$] at (0,0){};
\node (g)[point,label=right:$y_3$] at (1,0){};
\node (h)[point,label=right:$y_4$] at (2,0){};

\draw (a.center) -- (e.center);
\draw (b.center) -- (f.center);
\draw (c.center) -- (g.center);
\draw (d.center) -- (h.center);
\draw (a.center) -- (f.center);
\draw (a.center) -- (g.center);
\draw (a.center) -- (h.center);
\draw (b.center) -- (g.center);
\draw (b.center) -- (h.center);

	\filldraw [black] (a.center) circle (1pt);
	\filldraw [black] (b.center) circle (1pt);
	\filldraw [black] (c.center) circle (1pt);
	\filldraw [black] (d.center) circle (1pt);
	\filldraw [black] (e.center) circle (1pt);
	\filldraw [black] (f.center) circle (1pt);
	\filldraw [black] (g.center) circle (1pt);
	\filldraw [black] (h.center) circle (1pt);
\end{tikzpicture}
\end{minipage}
\begin{minipage}[c][6cm][c]{0.3\textwidth}
\begin{tikzpicture}[every loop/.style={}]
\tikzstyle{point}=[inner sep=0pt]
\node (a)[point,label=left: $x_1$] at (-1,1) {};
\node (b)[point,label=left:$x_2$] at (0,.3){};
\node (c)[point,label=right:$x_3$] at (1,0.3){};
\node (d)[point,label=right:$x_4$] at (2,1){};

\draw (a.center) -- (b.center);
\draw (a.center) -- (c.center);
\draw (a.center) -- (d.center);
\draw (b.center) -- (c.center);
\draw (b.center) -- (d.center);

\draw (a) edge [anchor=center, loop above] (a);
\draw (b) edge [anchor=center, loop below] (b);
\draw (c) edge [anchor=center, loop below] (c);
\draw (d) edge [anchor=center, loop above] (d);

	\filldraw [black] (a.center) circle (1pt);
	\filldraw [black] (b.center) circle (1pt);
	\filldraw [black] (c.center) circle (1pt);
	\filldraw [black] (d.center) circle (1pt);
	
\end{tikzpicture}
\end{minipage}
\begin{minipage}[c][6cm][c]{0.3\textwidth}
\begin{tikzpicture}[every loop/.style={}]
\tikzstyle{point}=[inner sep=0pt]
\node (a)[point,label=left: $x_1$] at (-1,.7) {};
\node (b)[point,label=left:$x_2$] at (0,0){};
\node (c)[point,label=right:$x_3$] at (1,0){};
\node (d)[point,label=right:$x_4$] at (2,.7){};
\node (e)[point,label=left:$y_1$] at (-1,1.5){};
\node (f)[point,label=left:$y_2$] at (0,1.2){};
\node (g)[point,label=right:$y_3$] at (1,1.2){};
\node (h)[point,label=right:$y_4$] at (2,1.5){};

\draw (a.center) -- (b.center);
\draw (a.center) -- (c.center);
\draw (a.center) -- (d.center);
\draw (b.center) -- (c.center);
\draw (b.center) -- (d.center);
\draw (a.center) -- (e.center);
\draw (b.center) -- (f.center);
\draw (c.center) -- (g.center);
\draw (d.center) -- (h.center);

	\filldraw [black] (a.center) circle (1pt);
	\filldraw [black] (b.center) circle (1pt);
	\filldraw [black] (c.center) circle (1pt);
	\filldraw [black] (d.center) circle (1pt);
	\filldraw [black] (e.center) circle (1pt);
	\filldraw [black] (f.center) circle (1pt);
	\filldraw [black] (g.center) circle (1pt);
	\filldraw [black] (h.center) circle (1pt);
\end{tikzpicture}
\end{minipage}
\end{minipage}
\end{example}
\vspace{-.5in}

The following illustrates how the results can be applied when the graph is not Cohen-Macaulay, and thus does not have a perfect matching.

\begin{example}
Let \(I=(x_1x_2,x_2x_3,x_3x_4,x_4x_5,x_5x_6,x_6x_7,x_2x_7,x_2x_5)\) be the edge ideal of the graph depicted below. Notice that \(\{x_1,x_2\}, \{x_3,x_4\},\{x_6,x_7\}\) is a maximal matching where each edge has Property \(P\) and satisfies the condition on neighbor sets neccesary to apply Theorem~\ref{reg seq}.
 As in \cite[Proposition 4.4]{Brennan-Morey}, the $h$-vector of each of these three graphs is computed via a difference of two $f$-vectors: one that counts independent sets of size $i$ and the other (shifted) counting independent sets of size $i-1$ that do not contain \(x_5\). The result is the $h$-vector \((1,3,-2,-1)\), written below as a difference of polynomials to make the shift clear:
\[(1+4t+t^2) - t(1+3t+t^2) = 1 + 3t -2t^2-t^3.\]
\noindent\begin{minipage}{\textwidth}
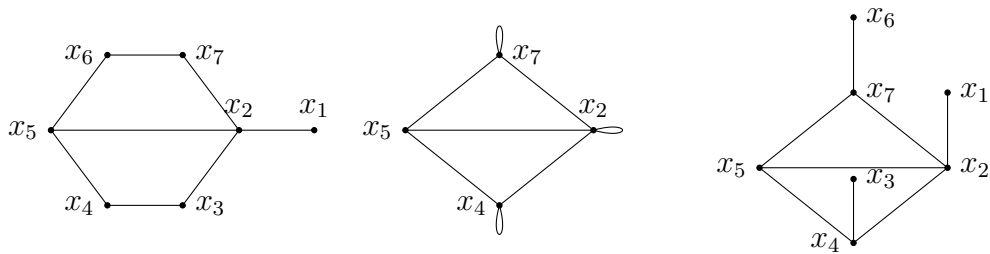
\begin{minipage}[c][6cm][c]{0.3\textwidth}
\begin{tikzpicture}
\tikzstyle{point}=[inner sep=0pt]
\node (a)[point,label=above: $x_1$] at (3.5, 1) {};
\node (b)[point,label=above:$x_2$] at (2.5,1){};
\node (c)[point,label=right:$x_3$] at (1.75,0){};
\node (d)[point,label=left:$x_4$] at (0.75,0){};
\node (e)[point,label=left:$x_5$] at (0,1){};
\node (f)[point,label=left:$x_6$] at (0.75,2){};
\node (g)[point,label=right:$x_7$] at (1.75,2){};

\draw (a.center) -- (b.center);
\draw (b.center) -- (c.center);
\draw (c.center) -- (d.center);
\draw (d.center) -- (e.center);
\draw (e.center) -- (f.center);
\draw (f.center) -- (g.center);
\draw (b.center) -- (g.center);
\draw (b.center) -- (e.center);

	\filldraw [black] (a.center) circle (1pt);
	\filldraw [black] (b.center) circle (1pt);
	\filldraw [black] (c.center) circle (1pt);
	\filldraw [black] (d.center) circle (1pt);
	\filldraw [black] (e.center) circle (1pt);
	\filldraw [black] (f.center) circle (1pt);
	\filldraw [black] (g.center) circle (1pt);
	
\end{tikzpicture}
\end{minipage}
\begin{minipage}[c][6cm][c]{0.3\textwidth}
\begin{tikzpicture}[every loop/.style={}]
\tikzstyle{point}=[inner sep=0pt]
\node (b)[point,label=above:$x_2$] at (2.5,1){};
\node (d)[point,label=left:$x_4$] at (1.25,0){};
\node (e)[point,label=left:$x_5$] at (0,1){};
\node (g)[point,label=right:$x_7$] at (1.25,2){};

\draw (b.center) -- (d.center);
\draw (d.center) -- (e.center);
\draw (e.center) -- (g.center);
\draw (b.center) -- (g.center);
\draw (b.center) -- (e.center);

\draw (b) edge [anchor=center, loop right] (b);
\draw (d) edge [anchor=center, loop below] (d);
\draw (g) edge [anchor=center, loop above] (g);

	\filldraw [black] (b.center) circle (1pt);
	\filldraw [black] (d.center) circle (1pt);
	\filldraw [black] (e.center) circle (1pt);
	\filldraw [black] (g.center) circle (1pt);
	
\end{tikzpicture}
\end{minipage}
\begin{minipage}[c][6cm][c]{0.3\textwidth}
\begin{tikzpicture}[every loop/.style={}]
\tikzstyle{point}=[inner sep=0pt]
\tikzstyle{point}=[inner sep=0pt]
\node (a)[point,label=right:$x_1$] at (2.5,2){};
\node (b)[point,label=right:$x_2$] at (2.5,1){};
\node (c)[point,label=right:$x_3$] at (1.25,0.85){};
\node (d)[point,label=left:$x_4$] at (1.25,0){};
\node (e)[point,label=left:$x_5$] at (0,1){};
\node (f)[point,label=right:$x_6$] at (1.25,3){};
\node (g)[point,label=right:$x_7$] at (1.25,2){};

\draw (b.center) -- (d.center);
\draw (d.center) -- (e.center);
\draw (e.center) -- (g.center);
\draw (b.center) -- (g.center);
\draw (b.center) -- (e.center);
\draw (a.center) -- (b.center);
\draw (c.center) -- (d.center);
\draw (f.center) -- (g.center);

	\filldraw [black] (a.center) circle (1pt);
	\filldraw [black] (b.center) circle (1pt);
	\filldraw [black] (c.center) circle (1pt);
	\filldraw [black] (d.center) circle (1pt);
	\filldraw [black] (e.center) circle (1pt);
	\filldraw [black] (f.center) circle (1pt);
 	\filldraw [black] (g.center) circle (1pt);
\end{tikzpicture}
\end{minipage}
\end{minipage}
\end{example}

We conclude with an example of how these techniques can be used when there is a known regular sequence consisting of linear binomials that do not necessarily all come from edges with Property \(P\). Examples of such regular elements can be found in \cite{FHM} and such binomials are classified in Theorem~\ref{binomial regular}.

\begin{example}
Let $I=(x_1x_2, x_2x_3,x_2x_4,x_4x_5,x_4x_6)$ be the edge ideal of the graph $G$ depicted below. Then $x_1+x_2, x_4+x_6, x_3+x_5$ forms a regular sequence on $R/I$ (see \cite{FHM}). The graph corresponding to $R/(I,x_1+x_2, x_4+x_6, x_3+x_5)$ and its polarization are also depicted below. Note that the $h$-vector can again be computed as a difference via \cite[Proposition 4.4]{Brennan-Morey}.

\noindent\begin{minipage}{\textwidth}
\begin{minipage}[c][6cm][c]{0.3\textwidth}
\begin{tikzpicture}
\tikzstyle{point}=[inner sep=0pt]
\node (a)[point,label=right: $x_1$] at (0,0) {};
\node (b)[point,label=left:$x_2$] at (1,1){};
\node (c)[point,label=right:$x_3$] at (0,2){};
\node (d)[point,label=right:$x_4$] at (2,1){};
\node (e)[point,label=left:$x_5$] at (3,2){};
\node (f)[point,label=left:$x_6$] at (3,0){};

\draw (a.center) -- (b.center);
\draw (b.center) -- (c.center);
\draw (b.center) -- (d.center);
\draw (d.center) -- (e.center);
\draw (d.center) -- (f.center);

	\filldraw [black] (a.center) circle (1pt);
	\filldraw [black] (b.center) circle (1pt);
	\filldraw [black] (c.center) circle (1pt);
	\filldraw [black] (d.center) circle (1pt);
	\filldraw [black] (e.center) circle (1pt);
	\filldraw [black] (f.center) circle (1pt);

\end{tikzpicture}
\end{minipage}
\begin{minipage}[c][6cm][c]{0.3\textwidth}
\begin{tikzpicture}[every loop/.style={}]
\tikzstyle{point}=[inner sep=0pt]
\node (b)[point,label=left:$x_1$] at (1,1){};
\node (c)[point,label=right:$x_3$] at (1.5,2){};
\node (d)[point,label=right:$x_4$] at (2,1){};

\draw (b.center) -- (c.center);
\draw (b.center) -- (d.center);
\draw (d.center) -- (c.center);

\draw (b) edge [anchor=center, loop below] (b);
\draw (d) edge [anchor=center, loop below] (d);

	\filldraw [black] (b.center) circle (1pt);
	\filldraw [black] (c.center) circle (1pt);
	\filldraw [black] (d.center) circle (1pt);
	
\end{tikzpicture}
\end{minipage}
\begin{minipage}[c][6cm][c]{0.3\textwidth}
\begin{tikzpicture}[every loop/.style={}]
\tikzstyle{point}=[inner sep=0pt]
\node (a)[point,label=right: $x_2$] at (0,0) {};
\node (b)[point,label=left:$x_1$] at (1,1){};
\node (c)[point,label=right:$x_3$] at (1.5,2){};
\node (d)[point,label=right:$x_4$] at (2,1){};
\node (f)[point,label=left:$x_6$] at (3,0){};

\draw (a.center) -- (b.center);
\draw (b.center) -- (c.center);
\draw (b.center) -- (d.center);
\draw (d.center) -- (c.center);
\draw (d.center) -- (f.center);

	\filldraw [black] (a.center) circle (1pt);
	\filldraw [black] (b.center) circle (1pt);
	\filldraw [black] (c.center) circle (1pt);
	\filldraw [black] (d.center) circle (1pt);
	\filldraw [black] (f.center) circle (1pt);
	
\end{tikzpicture}
\end{minipage}
\end{minipage}
Then as above, writing the $h$-vector as a polynomial in $t$ to emphasize the shifts, we have the numerator of the Hilbert series is $(1+3t) - t(1+2t) = 1+2t-2t^2$, where $1+3t$ encapsulates the independent sets of the triangle $x_1,x_3,x_4$ and $t(1+2t)$ counts independent sets not involving $x_3$.

\end{example}

\section{Declarations}
No external funding was received to assist with the preparation of this manuscript. The authors have no relevant financial or non-financial interests to disclose. There is no associated data with this manuscript.
 \bibliography{bibliography}
\end{document}